\documentclass[10pt,a4paper,reqno]{amsart}
\usepackage{amsmath,amssymb,amsthm,mathtools}
\usepackage{lmodern}
\usepackage[a4paper,left=1in,right=1in,top=1in,bottom=0.95in]{geometry}
\usepackage{bm}
\usepackage{microtype}
\usepackage{enumitem}
\usepackage{hyperref}
\usepackage[nameinlink,noabbrev]{cleveref}
\usepackage{mathrsfs}

\hypersetup{colorlinks=true,linkcolor=blue,citecolor=blue,urlcolor=blue}

\newtheorem{theorem}{Theorem}[section]
\newtheorem{lemma}[theorem]{Lemma}
\newtheorem{proposition}[theorem]{Proposition}
\newtheorem{corollary}[theorem]{Corollary}
\newtheorem{remark}[theorem]{Remark}

\newcommand{\R}{\mathbb R}
\newcommand{\Sn}{\mathbb S}

\newcommand{\cG}{\mathcal G}
\newcommand{\cB}{\mathcal B}
\newcommand{\cL}{\mathcal L}
\newcommand{\whf}{\widehat f}
\newcommand{\inn}[2]{\langle #1,#2\rangle}
\newcommand{\norm}[1]{\lVert #1\rVert}
\newcommand{\abs}[1]{\lvert #1\rvert}

\newcommand{\dist}{\operatorname{dist}}

\makeatletter
\def\@settitle{\begin{center}%
  \baselineskip17\p@\relax
  {\bfseries\Large\@title\par}%
  \end{center}%
}
\def\@setauthors{%
  \begingroup
  \def\thanks{\protect\thanks@warning}%
  \trivlist
  \centering\normalsize\bfseries
  \@topsep22\p@\relax
  \advance\@topsep by -\baselineskip
  \item\relax
  \author@andify\authors
  \def\\{\protect\linebreak}%
  \authors
  \endtrivlist
  \endgroup
}

\def\@adminfootnotes{%
  \let\@makefnmark\relax \let\@thefnmark\relax
  \ifx\@empty\@date\else \@footnotetext{\@setdate}\fi
  \ifx\@empty\@keywords\else \@footnotetext{\@setkeywords}\fi
  \ifx\@empty\@subjclass\else \@footnotetext{\@setsubjclass}\fi
  \ifx\@empty\thankses\else \@footnotetext{%
    \def\par{\let\par\@par}\@setthanks}%
  \fi
}
\def\@seccntformat#1{\csname the#1\endcsname\quad}
\renewcommand\section{\@startsection{section}{1}{\z@}%
  {-3.25ex \@plus -1ex \@minus -.2ex}%
  {1.5ex \@plus .2ex}%
  {\normalfont\large\bfseries}}
\renewcommand\subsection{\@startsection{subsection}{2}{\z@}%
  {-2.6ex\@plus -1ex \@minus -.2ex}%
  {0.9ex \@plus .2ex}%
  {\normalfont\normalsize\bfseries}}
\makeatother

\title[Comparison Pogorelov estimate]{A Comparison Pogorelov Estimate for Graphical $\sigma_k$-Curvature Equations}
\author{Jundong Zhou}
\address{School of Mathematics and Statistics, Fuyang Normal University, Fuyang 236037, Anhui, People's Republic of China}
\email{zhou109@mail.ustc.edu.cn}
\date{}

\begin{document}

\maketitle

\begin{center}
{\small\bfseries Abstract}
\end{center}
\vspace{0pt}
\begin{list}{}{\leftmargin=3pc\rightmargin=3pc\topsep=0pt\partopsep=0pt\parsep=0pt\itemsep=0pt}
\item[]\small
We establish comparison Pogorelov estimates for admissible solutions of graphical $\sigma_k$-curvature equations, extending the two-surface estimate of Qiu and Yan for the graphical scalar curvature equation. The comparison graph is assumed only to be $k$-admissible, with a bounded slope and a positive interior gap that vanishes on the boundary. The estimates cover $2\le k<n$ for prescribed data $f(x,u)$, and $n/2\le k<n$ for general data $f(x,u,Du)$. The proof combines the two-surface localization and mixed G\aa rding comparison of Qiu and Yan with Yan's spectral concavity inequality and an angle-function weight.
\end{list}
\vspace{0.35em}

\section{Introduction}

Let $\Omega\subset\R^n$, $n\ge3$, be a bounded domain and let
\[
M_v=\{(x,v(x)):x\in\Omega\}\subset\R^{n+1}
\]
be the graph of a function $v$.  We use the downward unit normal
\[
\nu_v=\frac{(Dv,-1)}{W_v},\qquad W_v=\sqrt{1+\abs{Dv}^2},
\]
and the convention $\nabla_e\nu_v=A_v(e)$.  If $g_v=I+Dv\otimes Dv$, then the Weingarten operator is represented, under the canonical orthonormal identification of $\R^n$ with $T_{(x,v(x))}M_v$, by
\begin{equation*}
A_v=\frac1{W_v}\,g_v^{-1/2}D^2v\,g_v^{-1/2}.
\end{equation*}
We denote its principal curvatures by $\kappa_1[v]\ge\cdots\ge\kappa_n[v]$.  As usual,
\[
\Gamma_k=\{\lambda\in\R^n:\sigma_j(\lambda)>0,\ 1\le j\le k\},
\]
and a graph is called $k$-admissible when $\kappa[v]\in\Gamma_k$.  The basic hyperbolic-polynomial structure of these cones is due to G\aa rding \cite{Garding1959}.

The aim of this paper is to develop the Pogorelov estimate for $k$-Hessian equations into a comparison estimate for graphical $\sigma_k$-curvature equations.  Pogorelov's method for the Monge--Amp\`ere equation is one of the basic sources of interior second-derivative estimates in fully nonlinear elliptic theory.  The Dirichlet theory for Hessian and Weingarten equations was developed in the work of Caffarelli--Nirenberg--Spruck, Ivochkina and Trudinger; see \cite{CNS1984,CNS1985,CNS1986,CNS1988,Ivochkina1990,Trudinger1990}.  Chou and Wang obtained the corresponding weighted Pogorelov estimate for $k$-Hessian equations \cite{ChouWang2001}.  In the geometric setting, Sheng, Urbas and Wang proved interior curvature bounds for graphical curvature equations with affine Dirichlet data \cite{ShengUrbasWang2004}.  Further curvature and Pogorelov-type estimates were obtained in, among others, \cite{Urbas2000,ShengTrudingerWang2004,GuanRenWang2015,LiRenWang2019,RenWang2019,RenWang2023,Lu2023,LuTsai2026,ChenTuXiang2021,ChenLuTuXiang2021,LiuRen2023,JiaoJiao2024}.  Urbas' singular examples \cite{Urbas1990} also show why purely interior second-order estimates depend strongly on the structure of the operator.  For related developments around scalar curvature and the $\sigma_2$ equation, we refer to \cite{WarrenYuan2009,GuanQiu2019,SpruckXiao2017,ShankarYuan2020,Qiu2024Curvature,Qiu2024Hessian,ShankarYuan2025,Fan2026} and the references therein.

A basic obstacle in extending the classical Pogorelov estimate to graphical curvature equations is the comparison graph.  The affine comparison functions used in Chou--Wang \cite{ChouWang2001} and the affine boundary data in Sheng--Urbas--Wang \cite{ShengUrbasWang2004} avoid the mismatch between the tangent spaces and normals of two general graphs.  Qiu and Yan \cite[Theorem~4.5]{QiuYan2026} overcame this difficulty for the graphical scalar curvature equation.  They established a comparison Pogorelov estimate with an arbitrary $\Gamma_2$-admissible upper graph, without requiring that graph to satisfy a second curvature equation or imposing a bound on its second derivatives.

The present paper extends their two-surface comparison framework to graphical $\sigma_k$-curvature equations with variable prescribed data.  We retain the quadratic product-space localization, the identification of the tangent spaces by the natural rotation, and the mixed G\aa rding comparison used in \cite[Section~4.2]{QiuYan2026}.  The additional estimates concern the largest principal curvature for general $k$ and the terms arising from the variable right-hand side.  These are obtained using Yan's spectral concavity inequality \cite{Yan2026} together with an angle-function weight.  The admissible-jet example in \cite[Remark~3.3]{QiuYan2026} explains why a direct one-point comparison of the vertical gap is insufficient, even when $k=2$ and the slopes are uniformly bounded.

For completeness, in the $k=2$ example of Qiu and Yan the two endpoint linearizations of the vertical gap may have arbitrarily large values with opposite signs,
\[
T_u^{ij}\nabla^u_{ij}(w-u)\to-\infty,
\qquad
T_w^{ij}\nabla^w_{ij}(w-u)\to+\infty.
\]
No higher-order analogue of this example is needed here; its role is to identify the geometric obstruction to a fixed-coordinate one-point argument.  For general $k$, the largest-curvature calculation adds a third-order spectral quadratic form, and a gradient-dependent prescribed function also produces normal-derivative terms.  Yan's inequality controls the spectral part, while the angle-function term provides the positive quadratic curvature contribution needed to absorb the remaining loss.  This gives the full range $2\le k<n$ for $f(x,u)$ and the supercritical range $n/2\le k<n$ for $f(x,u,Du)$.

We state the two estimates under a common set of assumptions.  Let
\[
u\in C^4(\Omega)\cap C^1(\overline\Omega),\qquad
w\in C^2(\Omega)\cap C^1(\overline\Omega),\qquad A_u,A_w\in\Gamma_k,
\]
and suppose
\[
w>u\ \text{in }\Omega,\qquad w=u\ \text{on }\partial\Omega,
\]
with
\[
\norm{u}_{L^\infty(\Omega)}\le U,
\qquad
\norm{Du}_{L^\infty(\Omega)}+\norm{Dw}_{L^\infty(\Omega)}\le M.
\]
We also assume that $\kappa_1[u]>\kappa_2[u]$ in $\Omega$ and set $\rho=w-u$.  The prescribed function is positive and $C^{1,1}$ on a fixed neighborhood of the relevant range, with $0<f_0\le f\le f_1$.

\begin{theorem}\label{thm:main}
Let $n\ge3$ and $2\le k<n$.  If
\begin{equation*}
\sigma_k(A_u)=f(x,u)\quad\text{in }\Omega,
\end{equation*}
then there are $\beta=\beta(n,k,M)>1$ and $C<\infty$ such that
\begin{equation*}
\sup_\Omega \rho^\beta\abs{A_u}\le C.
\end{equation*}
The constant depends only on $n,k,M,U,\norm{\rho}_{L^\infty(\Omega)},f_0^{-1},f_1,\norm{f}_{C^{1,1}}$ and the fixed coordinate range.  It is independent of $\norm{D^2w}_{L^\infty}$, of any positive lower bound for $\sigma_k(A_w)$, and of the convexity of $\Omega$.  In particular, the result applies to $f=f(x)$.
\end{theorem}

\begin{theorem}\label{thm:gradient}
Let $n\ge3$ and $n/2\le k<n$.  If
\begin{equation*}
\sigma_k(A_u)=f(x,u,Du)\quad\text{in }\Omega,
\end{equation*}
then
\[
\sup_\Omega \rho^\beta\abs{A_u}\le C
\]
for some $\beta>1$ and $C<\infty$ depending only on the data listed in \cref{thm:main}.
\end{theorem}

The proofs are given in Subsection~3.2.

\begin{remark}\label{rem:intro}
The strict gap between the two graphs is assumed, not derived.  Constructing the comparison graph from boundary data or from a subsolution is a separate problem.  The graph $w$ need not solve a curvature equation, but its full admissibility $A_w\in\Gamma_k$ is essential for the mixed G\aa rding inequality.  The simple-root assumption is used only to keep the largest eigenvalue smooth in the pointwise calculation.  The restriction $k<n$ comes from the spectral inequality quoted in Section~2.
\end{remark}

For the curvature part of the argument, the largest principal curvature is coupled with a power of the angle function.  Yan's inequality gives the Jacobi estimate, and the angle term contributes a positive quadratic curvature term.  The weak spectral directions are precisely those in which the corresponding principal curvatures are comparable with the largest one, so this quadratic term absorbs them.  The $C^{1,1}$ hypothesis on the prescribed function enters only through a one-sided second-order estimate along graph geodesics.

The comparison is made on the product of the two graphs.  A quadratic translation cost places the maximum at an interior pair.  The natural rotation then identifies the two tangent spaces, and the second variation involves the mixed contraction of the linearization at $A_u$ with the rotated second fundamental form of $w$.  Concavity and homogeneity of $\sigma_k^{1/k}$ give the required sign without any estimate for $D^2w$.

Section~2 contains the algebraic and differential inequalities used in the argument.  Section~3 gives the two-surface localization and completes the maximum-principle proof of the two theorems.

\section{Preliminaries and key estimates}

\subsection{Cone estimates and Yan's spectral inequality}

All quantities in this subsection are evaluated on the graph of $u$. At a point where $A_u$ is diagonal, write
\begin{equation*}
a=\kappa_1>\kappa_2\ge\cdots\ge\kappa_n,\qquad
F=\sigma_k(\kappa)=f,
\end{equation*}
and
\begin{equation*}
F^{ii}=\sigma_{k-1}(\kappa|i),\qquad
\cL=\sum_iF^{ii}\nabla_{ii},\qquad
Q=\sum_iF^{ii}\kappa_i^2.
\end{equation*}
For $i\ne j$, set
\[
F^{ii,jj}=\sigma_{k-2}(\kappa|ij),\qquad F^{ii,ii}=0.
\]

We begin with the cone estimates used throughout the proof.  Besides ellipticity, they give a lower bound for the quadratic curvature term $Q$ and a positive lower bound for $a$.  The bound for $Q$ is used later to absorb the bad spectral directions.

\begin{lemma}\label{lem:cone}
If $\kappa\in\Gamma_k$ and $2\le k<n$, then
\begin{equation}
F^{ii}>0,\qquad \sum_iF^{ii}\kappa_i=kf,\qquad
\kappa_i> -\frac{n-k}{k}a.
\label{eq:cone-basic}
\end{equation}
Moreover, with $H=\sigma_1(\kappa)$,
\begin{equation}
Q\ge\frac{k}{n}Hf\ge\frac{k}{n}f_0a,
\qquad \abs{A_u}<H\le na,
\label{eq:Q-lower}
\end{equation}
and
\begin{equation}
a\ge \left(\frac{f_0}{\binom nk}\right)^{1/k}>0.
\label{eq:a-lower}
\end{equation}
\end{lemma}

\begin{proof}
Since $\kappa\in\Gamma_k$, G\aa rding cone theory implies that the linearization of $\sigma_k$ is elliptic \cite{Garding1959}; hence
\[
F^{ii}=\sigma_{k-1}(\kappa|i)>0.
\]
Euler's identity for the homogeneous polynomial $\sigma_k$ gives
\[
\sum_iF^{ii}\kappa_i=k\sigma_k(\kappa)=kf.
\]
For the individual principal curvatures it is enough to consider the smallest component $\kappa_n$, since $\kappa_i\ge\kappa_n$ for every $i$.  If $\kappa_n\ge0$, there is nothing to prove.  Suppose therefore that $\kappa_n<0$ and set $\mu=\kappa|n\in\R^{n-1}$.  We claim that $\mu\in\Gamma_k$.  Indeed, starting from $\sigma_0(\mu)=1$ and using
\[
\sigma_j(\kappa)=\sigma_j(\mu)+\kappa_n\sigma_{j-1}(\mu),
\]
an induction on $j$ gives $\sigma_j(\mu)>0$ for $1\le j\le k$, because $\sigma_j(\kappa)>0$ and $\kappa_n<0$.  Consequently
\[
0<\sigma_k(\kappa)=\sigma_k(\mu)+\kappa_n\sigma_{k-1}(\mu),
\]
so
\[
-\kappa_n<\frac{\sigma_k(\mu)}{\sigma_{k-1}(\mu)}.
\]
Since $\mu_i\le a$ and the first derivatives of $\sigma_k$ are positive on $\Gamma_k$ \cite{Garding1959},
\[
k\sigma_k(\mu)
 =\sum_{i=1}^{n-1}\mu_i\sigma_{k-1}(\mu|i)
 \le a\sum_{i=1}^{n-1}\sigma_{k-1}(\mu|i)
 =a(n-k)\sigma_{k-1}(\mu).
\]
Therefore
\[
\kappa_n> -\frac{n-k}{k}a,
\]
and the same estimate holds for every $\kappa_i$.  This proves \eqref{eq:cone-basic}.

Because $k\ge2$, one has $\kappa\in\Gamma_2$.  Thus $H=\sigma_1(\kappa)>0$ and
\[
H^2-\abs{A_u}^2=2\sigma_2(\kappa)>0,
\]
which yields $\abs{A_u}<H$.  Since $a=\max_i\kappa_i$, we also have $H=\sum_i\kappa_i\le na$.

For $Q$ we use the identity
\[
\sigma_1\sigma_k=(k+1)\sigma_{k+1}
 +\sum_i\kappa_i^2\sigma_{k-1}(\kappa|i).
\]
The last sum is precisely $Q$, and hence
\[
Q=H\sigma_k-(k+1)\sigma_{k+1}.
\]
If $\sigma_{k+1}\le0$, then $Q\ge Hf$, which is stronger than the desired estimate.  If $\sigma_{k+1}>0$, then $\kappa\in\Gamma_{k+1}$.  For the normalized elementary symmetric functions
\[
S_j=\frac{\sigma_j}{\binom nj},
\]
Newton--Maclaurin implies that the ratios $S_{j+1}/S_j$ are nonincreasing.  In particular,
\[
\frac{S_{k+1}}{S_k}\le S_1=\frac{H}{n}.
\]
Since $\binom n{k+1}/\binom nk=(n-k)/(k+1)$, this gives
\[
(k+1)\sigma_{k+1}\le\frac{n-k}{n}H\sigma_k.
\]
Substitution into the preceding identity yields
\[
Q\ge\frac{k}{n}Hf.
\]
Finally, the normalized Maclaurin inequality gives
\[
\left(\frac{f}{\binom nk}\right)^{1/k}
 =S_k^{1/k}\le S_1=\frac{H}{n}\le a.
\]
Using $f\ge f_0$ proves \eqref{eq:Q-lower} and \eqref{eq:a-lower}.
\end{proof}

The preceding cone estimates control the lower-order curvature terms, but they do not address the third-order quadratic form produced by differentiating the largest principal curvature.  We use Yan's spectral concavity inequality.  For $\lambda\in\Gamma_k$, let
\[
a=\lambda_1>\lambda_2\ge\cdots\ge\lambda_n,\qquad F=\sigma_k(\lambda),
\]
with $F^{ii}=\sigma_{k-1}(\lambda|i)$ and $F^{ii,jj}=\sigma_{k-2}(\lambda|ij)$ for $i\ne j$. For $\xi\in\R^n$, define
\begin{align}
Q_\gamma(\lambda;\xi)
={}&\frac{2}{aF}\left(\sum_iF^{ii}\xi_i\right)^2
-\gamma\frac{F^{11}}{a^2}\xi_1^2
-\frac1a\sum_{i,j}F^{ii,jj}\xi_i\xi_j \notag\\
&\quad +\frac2a\sum_{p=2}^n\frac{F^{pp}}{a-\lambda_p}\xi_p^2.
\label{eq:Qgamma}
\end{align}

\begin{theorem}[\cite{Yan2026}, Theorem~1.1]\label{thm:yan}
Let $n\ge3$, $2\le k<n$, and
\begin{equation}
0<\gamma<\min\left\{\frac{2k}{n},\ 1+\frac{2k-n}{2k^2+n}\right\}.
\label{eq:gamma-range}
\end{equation}
There exists $\eta_Y=\eta_Y(n,k,\gamma)>0$ such that
\begin{equation}
0<\frac{F}{a^k}\le\eta_Y
\label{eq:small-level}
\end{equation}
implies
\begin{equation*}
Q_\gamma(\lambda;\xi)\ge0\qquad\text{for every }\xi\in\R^n.
\end{equation*}
\end{theorem}

The small-level condition is used only in the large-curvature branch. If \eqref{eq:small-level} fails along the equation, then
\begin{equation}
a<\left(\frac{f_1}{\eta_Y}\right)^{1/k},
\label{eq:smalllevel-fail}
\end{equation}
so the desired weighted bound is immediate.  Hence all subsequent Jacobi calculations may be carried out in the branch where \eqref{eq:small-level} holds and \cref{thm:yan} is available.

\subsection{Jacobi inequalities}

We derive the differential inequality for the largest curvature.  Since the prescribed function is only $C^{1,1}$, its second derivative is handled by a one-sided chain rule.  The resulting Jacobi inequality is then combined with the graph angle identity, which leaves a positive quadratic curvature term.

The equation must be differentiated twice along a principal direction although the prescribed function is only $C^{1,1}$.  A one-sided chain rule is sufficient.

For a gradient-dependent right-hand side, define on the relevant lower spherical cap
\begin{equation*}
\whf(X,\nu)=f\left(x,z,-\frac{\nu'}{\nu_{n+1}}\right),\qquad X=(x,z)\in\R^{n+1}.
\end{equation*}
The slope bound keeps $-\nu_{n+1}$ uniformly positive, so $\whf$ is $C^{1,1}$ on a fixed compact subset of $\R^{n+1}\times\Sn^n$.

\begin{lemma}\label{lem:semiconvex}
Let $\Phi\in C^{1,1}$ on a coordinate neighborhood of a Riemannian manifold, and let $Y(t)$ be a $C^2$ curve. Assume that $g=\Phi\circ Y$ is $C^2$. Then
\begin{equation*}
g''(0)\ge D\Phi(Y(0))[\nabla_tY'(0)]-C\abs{Y'(0)}^2,
\end{equation*}
where $C$ depends only on the Lipschitz constant of $D\Phi$ and the chosen compact coordinate neighborhood.
\end{lemma}

\begin{proof}
Choose normal coordinates centered at $Y(0)$ and identify the coordinate neighborhood with a ball in Euclidean space.  Let $L$ be a Lipschitz constant for $D\Phi$ on a slightly smaller compact neighborhood containing the image of $Y(t)$ for $|t|$ small.  The fundamental theorem of calculus gives, for every small vector $z$,
\[
\Phi(z)-\Phi(0)-D\Phi(0)[z]
 =\int_0^1\big(D\Phi(sz)-D\Phi(0)\big)[z] \,ds,
\]
and therefore
\[
\Phi(z)\ge \Phi(0)+D\Phi(0)[z]-\frac L2\abs{z}^2.
\]
Write the coordinate representation of the curve as $y(t)$.  Since the Christoffel symbols vanish at the center of normal coordinates,
\[
y(t)=tV+\frac{t^2}{2}A+o(t^2),
\qquad
V=Y'(0),\quad A=\nabla_tY'(0).
\]
Consequently
\[
\abs{y(t)}^2=t^2\abs V^2+o(t^2).
\]
Applying the preceding semiconvex Taylor inequality with $z=y(t)$ gives
\[
g(t)\ge g(0)+tD\Phi(V)
 +\frac{t^2}{2}\big(D\Phi(A)-L\abs V^2\big)+o(t^2).
\]
On the other hand, $g$ is $C^2$, so
\[
g(t)=g(0)+tg'(0)+\frac{t^2}{2}g''(0)+o(t^2),
\qquad g'(0)=D\Phi(V).
\]
Comparing the quadratic terms yields
\[
g''(0)\ge D\Phi(A)-L\abs V^2.
\]
This gives the stated estimate, after enlarging the constant to account for the fixed coordinate chart.
\end{proof}

Choose a unit-speed geodesic $X(t)$ on the graph of $u$ with $X'(0)=e_1$, where $e_1$ is the principal direction corresponding to $a$. For
\[
Y(t)=(X(t),\nu(t))\in\R^{n+1}\times\Sn^n,
\]
the Gauss, Weingarten, and Codazzi equations imply
\begin{equation}
Y'(0)=(e_1,ae_1),\qquad
\nabla_tY'(0)=\left(-a\nu,\sum_{\ell=1}^nh_{11;\ell}e_\ell\right).
\label{eq:Y-derivs}
\end{equation}
Because $u\in C^4$, the composite $\whf(X(t),\nu(t))=\sigma_k(A_u(X(t)))$ is $C^2$, even though $\whf$ is only $C^{1,1}$.  The next estimate is the form of the differentiated equation needed in the largest-curvature calculation.

\begin{corollary}\label{cor:source-second}
At a principal frame,
\begin{equation*}
\sum_iF^{ii}h_{ii;11}+\sigma_k^{pq,rs}h_{pq;1}h_{rs;1}
\ge -aD_X\whf(\nu)+\sum_\ell h_{11;\ell}D_\nu\whf(e_\ell)-C(1+a^2).
\end{equation*}
If $f=f(x,u)$, then
\begin{equation*}
\sum_iF^{ii}h_{ii;11}+\sigma_k^{pq,rs}h_{pq;1}h_{rs;1}
\ge -aD_X\whf(\nu)-C.
\end{equation*}
In both cases,
\begin{equation}
\sum_iF^{ii}h_{ii;j}=D_X\whf(e_j)+\kappa_jD_\nu\whf(e_j).
\label{eq:source-first}
\end{equation}
\end{corollary}

\begin{proof}
Let $X(t)$ be the unit-speed graph geodesic chosen above, and parallel transport an orthonormal frame along $X(t)$ so that it agrees with the principal frame at $t=0$.  Set
\[
g(t)=\whf(X(t),\nu(t))=\sigma_k(A_u(X(t))).
\]
Since $u\in C^4$, the second fundamental form is $C^2$ and hence $g$ is $C^2$.  Differentiating the right-hand side at $t=0$ gives
\[
g''(0)=\sum_iF^{ii}h_{ii;11}
 +\sigma_k^{pq,rs}h_{pq;1}h_{rs;1}.
\]
For the geometric curve $Y(t)=(X(t),\nu(t))$, formula \eqref{eq:Y-derivs} gives
\[
Y'(0)=(e_1,ae_1),
\qquad
\nabla_tY'(0)=\left(-a\nu,\sum_\ell h_{11;\ell}e_\ell\right).
\]
Applying \cref{lem:semiconvex} to $\Phi=\whf$ therefore yields
\begin{align*}
g''(0)
&\ge -aD_X\whf(\nu)
 +\sum_\ell h_{11;\ell}D_\nu\whf(e_\ell)
 -C\big(\abs{e_1}^2+a^2\abs{e_1}^2\big)\\
&= -aD_X\whf(\nu)
 +\sum_\ell h_{11;\ell}D_\nu\whf(e_\ell)-C(1+a^2),
\end{align*}
which is the first assertion.

If $f=f(x,u)$, then $\whf$ depends only on the ambient position $X$.  We apply the same one-sided Taylor estimate to the curve $X(t)$ alone.  Since $\abs{X'(0)}=1$ and $\nabla_tX'(0)=-a\nu$, the semiconvex error is now bounded by a fixed constant rather than by $C(1+a^2)$.  Hence
\[
g''(0)\ge -aD_X\whf(\nu)-C,
\]
which proves the second assertion.

For the first derivative, differentiate the equation in an arbitrary principal direction $e_j$.  At the diagonal point,
\[
\frac{d}{ds}\sigma_k(A_u)=\sum_iF^{ii}h_{ii;j}.
\]
On the other hand,
\[
\frac{d}{ds}\whf(X,\nu)
 =D_X\whf(e_j)+D_\nu\whf(\nabla_{e_j}\nu)
 =D_X\whf(e_j)+\kappa_jD_\nu\whf(e_j),
\]
because $\nabla_{e_j}\nu=\kappa_j e_j$ in the principal frame.  This is \eqref{eq:source-first}.
\end{proof}

The Pogorelov test function is built from the largest principal curvature.  Since the final estimate is an upper bound for the largest principal curvature, we set
\[
\ell=\log a.
\]
The simple-root assumption makes $a$ smooth near every point. In a principal frame extended by geodesic coordinates,
\begin{equation}
a_{;ii}=h_{11;ii}+2\sum_{p=2}^n\frac{h_{1p;i}^2}{a-\kappa_p},
\qquad
h_{11;ii}=h_{ii;11}+a\kappa_i(a-\kappa_i).
\label{eq:eigenvalue-second}
\end{equation}

\begin{lemma}\label{lem:jacobi}
Assume \eqref{eq:small-level}, and define
\begin{equation*}
\theta_1=\gamma-1,\qquad
\theta_i=\frac{a+\kappa_i}{a-\kappa_i}\quad (i\ge2).
\end{equation*}
For a general right-hand side $f(x,u,Du)$,
\begin{equation}
\cL\ell\ge kfa-Q+\sum_iF^{ii}\theta_i\ell_i^2+D_\nu\whf(\nabla\ell)-C(1+a).
\label{eq:jacobi-general}
\end{equation}
If $f=f(x,u)$, then
\begin{equation}
\cL\ell\ge kfa-Q+\sum_iF^{ii}\theta_i\ell_i^2-C.
\label{eq:jacobi-position}
\end{equation}
\end{lemma}

\begin{proof}
All computations are made at the fixed point where the frame is orthonormal, geodesic, and diagonalizes $h$.  Put
\[
\xi_i=h_{ii;1}.
\]
The simple-eigenvalue formula in \eqref{eq:eigenvalue-second} gives
\begin{align*}
\cL a
&=\sum_iF^{ii}h_{11;ii}
 +2\sum_i\sum_{p=2}^n\frac{F^{ii}}{a-\kappa_p}h_{1p;i}^2\\
&=\sum_iF^{ii}h_{ii;11}
 +a\sum_iF^{ii}\kappa_i(a-\kappa_i)
 +2\sum_i\sum_{p=2}^n\frac{F^{ii}}{a-\kappa_p}h_{1p;i}^2.
\end{align*}
Using $\sum_iF^{ii}\kappa_i=kf$ and the definition of $Q$,
\[
a\sum_iF^{ii}\kappa_i(a-\kappa_i)=a^2kf-aQ.
\]
Therefore
\begin{align*}
\cL\ell
={}&kfa-Q+\frac1a\sum_iF^{ii}h_{ii;11}
 +\frac2a\sum_i\sum_{p=2}^n\frac{F^{ii}}{a-\kappa_p}h_{1p;i}^2
 -\frac1{a^2}\sum_iF^{ii}h_{11;i}^2.
\end{align*}
Apply \cref{cor:source-second} to the third term.  At a diagonal matrix the second derivative of $\sigma_k$ satisfies
\[
\sigma_k^{ij,rs}B_{ij}B_{rs}
 =\sum_{p\ne q}F^{pp,qq}B_{pp}B_{qq}
 -2\sum_{p<q}F^{pp,qq}B_{pq}^2.
\]
For $\kappa\in\Gamma_k$, the quantities $F^{pp,qq}=\sigma_{k-2}(\kappa|pq)$ are positive by the G\aa rding cone structure \cite{Garding1959}.  Hence the off-diagonal part contributes nonnegative terms after the preceding expression is inserted with a minus sign.

Separate the terms according to their indices.  Codazzi gives
\[
h_{1p;1}=h_{11;p}=a\ell_p,
\qquad
h_{1p;p}=h_{pp;1}=\xi_p,
\qquad
h_{11;1}=\xi_1=a\ell_1.
\]
The positive off-diagonal part coming from $-a^{-1}\sigma_k^{ij,rs}h_{ij;1}h_{rs;1}$ contains
\[
\frac2a\sum_{p=2}^nF^{11,pp}h_{1p;1}^2
 =2a\sum_{p=2}^nF^{11,pp}\ell_p^2,
\]
while the eigenvalue-variation term in the preceding identity contributes
\[
2a\sum_{p=2}^n\frac{F^{11}}{a-\kappa_p}\ell_p^2
 +\frac2a\sum_{p=2}^n\frac{F^{pp}}{a-\kappa_p}\xi_p^2.
\]
All terms with $2\le i\le n$, $i\ne p$, together with the off-diagonal terms indexed by $2\le p<q\le n$, are nonnegative and may be discarded.  Finally,
\[
-\frac1{a^2}\sum_iF^{ii}h_{11;i}^2
 =-\frac{F^{11}}{a^2}\xi_1^2-
 \sum_{p=2}^nF^{pp}\ell_p^2.
\]
For $p\ge2$, the total coefficient of $\ell_p^2$ is therefore
\begin{align*}
2aF^{11,pp}+\frac{2aF^{11}}{a-\kappa_p}-F^{pp}
&=\frac{2aF^{pp}}{a-\kappa_p}-F^{pp}\\
&=F^{pp}\frac{a+\kappa_p}{a-\kappa_p}.
\end{align*}
In the first equality we used
\[
F^{pp}-F^{11}=(a-\kappa_p)F^{11,pp}.
\]

At this point the remaining third-order expression is exactly the one appearing in $Q_\gamma$.  Add and subtract
\[
\frac{2}{af}\left(\sum_iF^{ii}\xi_i\right)^2
\]
and split
\[
-\frac{F^{11}}{a^2}\xi_1^2
 =-\gamma\frac{F^{11}}{a^2}\xi_1^2
 +(\gamma-1)F^{11}\ell_1^2.
\]
Using the definition \eqref{eq:Qgamma}, we obtain
\begin{align*}
\cL\ell\ge{}&kfa-Q+Q_\gamma(\kappa;\xi)
 +(\gamma-1)F^{11}\ell_1^2
 +\sum_{p=2}^n\frac{a+\kappa_p}{a-\kappa_p}F^{pp}\ell_p^2\\
&\quad -\frac{2}{af}\left(\sum_iF^{ii}\xi_i\right)^2
 +D_\nu\whf(\nabla\ell)-C(1+a)
\end{align*}
in the gradient-dependent case.  In the position-dependent case the normal drift is absent and the last error improves to $-C$ after division by $a$ and use of \eqref{eq:a-lower}.

It remains to control the explicit negative square.  By \eqref{eq:source-first},
\[
\sum_iF^{ii}\xi_i=D_X\whf(e_1)+aD_\nu\whf(e_1).
\]
The $C^1$ norm of $\whf$ is bounded on the fixed compact range, while $f\ge f_0$.  Thus, in the general case,
\[
\frac{2}{af}\left(\sum_iF^{ii}\xi_i\right)^2
 \le C\frac{(1+a)^2}{a}
 \le C(1+a),
\]
where the last inequality uses the fixed positive lower bound \eqref{eq:a-lower}.  If $f=f(x,u)$, then $D_\nu\whf=0$, and the same term is bounded by $C/a\le C$.

Finally, \eqref{eq:small-level} and \cref{thm:yan} imply $Q_\gamma(\kappa;\xi)\ge0$.  Recalling the definition of $\theta_i$ gives \eqref{eq:jacobi-general} and \eqref{eq:jacobi-position}.
\end{proof}

The Jacobi inequality for $\ell$ still contains the unfavorable term $-Q$ and, in the gradient-dependent case, a normal drift.  The graph structure supplies a complementary identity through the vertical angle function.  It is this identity that restores a positive multiple of $Q$ after the two estimates are combined.

Let $E=-e_{n+1}$ and
\[
\varphi=\inn{\nu_u}{E}=W_u^{-1},\qquad
\omega=\log W_u=-\log\varphi.
\]
Then
\begin{equation*}
\omega_i=-W_u\kappa_i\inn{E}{e_i}.
\end{equation*}
For a covector $\zeta$, write
\[
\abs{\zeta}_F^2=\sum_iF^{ii}\zeta_i^2.
\]

\begin{lemma}\label{lem:angle}
For a general right-hand side,
\begin{equation*}
\cL\omega=Q+\abs{\nabla\omega}_F^2+D_\nu\whf(\nabla\omega)-W_uD_X\whf(E^T),
\end{equation*}
where $E^T$ is the tangential projection of $E$. If $f=f(x,u)$, the $D_\nu\whf$ term is absent and the last term is uniformly bounded.
\end{lemma}

\begin{proof}
At the point under consideration choose the principal geodesic frame.  Since $E$ is a constant ambient vector and $\nabla_{e_i}\nu=\kappa_i e_i$,
\[
\varphi_i=\kappa_i\inn{E}{e_i},
\qquad
\omega_i=-\frac{\varphi_i}{\varphi}
 =-W_u\kappa_i\inn{E}{e_i}.
\]
Differentiating $\varphi_i$ once more and using that the tangential connection coefficients vanish at the point gives
\[
\varphi_{ii}=h_{ii;i}\inn{E}{e_i}-\kappa_i^2\varphi.
\]
After contraction with $F^{ii}$,
\[
\cL\varphi
 =\sum_iF^{ii}h_{ii;i}\inn{E}{e_i}-\varphi Q.
\]
To rewrite the first term, decompose the tangential projection of $E$ as
\[
E^T=\sum_j\inn{E}{e_j}e_j.
\]
Multiply \eqref{eq:source-first} by $\inn{E}{e_j}$ and sum over $j$.  This yields
\begin{align*}
\sum_iF^{ii}h_{ii;i}\inn{E}{e_i}
&=D_X\whf(E^T)
 +\sum_j\kappa_j\inn{E}{e_j}D_\nu\whf(e_j)\\
&=D_X\whf(E^T)-\varphi D_\nu\whf(\nabla\omega),
\end{align*}
because $\omega_j=-\kappa_j\inn{E}{e_j}/\varphi$.  Therefore
\[
\cL\varphi=D_X\whf(E^T)-\varphi Q-
\varphi D_\nu\whf(\nabla\omega).
\]
Now use $\omega=-\log\varphi$:
\[
\cL\omega
=-\frac{\cL\varphi}{\varphi}
 +\frac1{\varphi^2}\sum_iF^{ii}\varphi_i^2
=Q+\abs{\nabla\omega}_F^2
 +D_\nu\whf(\nabla\omega)-W_uD_X\whf(E^T).
\]
If $f=f(x,u)$, then $D_\nu\whf=0$.  Moreover $\abs{E^T}\le1$, $W_u\le\sqrt{1+M^2}$, and $D_X\whf$ is bounded on the fixed coordinate range, so the last term is uniformly bounded.
\end{proof}

The two identities are now combined.  The coefficient $q>1$ is kept free at this stage: in the position-dependent case it will be chosen algebraically, while in the gradient-dependent case it must also compensate for the linear curvature loss.  Define
\begin{equation}
G=\ell+q\omega,\qquad Z=e^G=aW_u^q,\qquad G_i=\ell_i+q\omega_i.
\label{eq:GZ}
\end{equation}
Combining \cref{lem:jacobi,lem:angle} yields
\begin{equation}
\cL G\ge kfa+(q-1)Q+\sum_iF^{ii}(\theta_i\ell_i^2+q\omega_i^2)
+D_\nu\whf(\nabla G)-Ca-C(1+q)
\label{eq:LG-general}
\end{equation}
for a general right-hand side, and
\begin{equation}
\cL G\ge kfa+(q-1)Q+\sum_iF^{ii}(\theta_i\ell_i^2+q\omega_i^2)-C(1+q)
\label{eq:LG-position}
\end{equation}
when $f=f(x,u)$.

\subsection{Choice of weights and good--bad spectral directions}

It remains to choose $q$ and $\gamma$ so that the gradient terms in the preceding Jacobi inequalities are compatible with the positive curvature term $Q$.  The main position-dependent theorem admits a fixed algebraic choice throughout the full range $2\le k<n$.  For the gradient-dependent extension, $q$ must also absorb the linear loss in \eqref{eq:LG-general}; this is where the restriction $k\ge n/2$ enters.  After the parameters are fixed, a good--bad decomposition isolates the directions in which the gradient square is weak and shows that precisely those directions are controlled by $Q$.

\begin{lemma}\label{lem:weights-position}
For every $2\le k<n$, choose
\begin{equation}
\gamma=\frac{k}{n},\qquad q=\frac{2n-k}{2(n-k)}.
\label{eq:position-weights}
\end{equation}
Then \eqref{eq:gamma-range} holds, and
\begin{equation}
q-1=\frac{k}{2(n-k)}>0,
\qquad
1-q(1-\gamma)=\frac{k}{2n}>0.
\label{eq:position-margins}
\end{equation}
If $k<n/2$, then
\begin{equation}
1-q\frac{n-2k}{n}>0.
\label{eq:position-extra-margin}
\end{equation}
\end{lemma}

\begin{proof}
The first bound in \eqref{eq:gamma-range} is immediate because
\[
0<\gamma=\frac{k}{n}<\frac{2k}{n}.
\]
For the second bound, a direct calculation gives
\[
1+\frac{2k-n}{2k^2+n}-\frac{k}{n}
 =\frac{k\{2k(n-k)+n\}}{n(2k^2+n)}>0,
\]
so $\gamma$ lies in the admissible interval of \cref{thm:yan}.  From the definition of $q$,
\[
q-1=\frac{2n-k-2(n-k)}{2(n-k)}
 =\frac{k}{2(n-k)}>0.
\]
Also,
\begin{align*}
1-q(1-\gamma)
&=1-\frac{2n-k}{2(n-k)}\frac{n-k}{n}\\
&=\frac{k}{2n}>0.
\end{align*}
If $k<n/2$, then
\begin{align*}
1-q\frac{n-2k}{n}
&=1-\frac{(2n-k)(n-2k)}{2n(n-k)}\\
&=\frac{k(3n-2k)}{2n(n-k)}>0.
\end{align*}
These positive margins are used below when a coefficient $\theta_i$ is negative.
\end{proof}

\begin{lemma}\label{lem:weights-general}
Assume $n/2\le k<n$. There are $q>1$, an admissible $\gamma$ satisfying \eqref{eq:gamma-range}, and constants $c_0,\mu>0$ such that
\begin{equation}
kfa+(q-1)Q-Ca\ge c_0a+\mu Q.
\label{eq:linear-absorb}
\end{equation}
If $k>n/2$, every $\theta_i$ has a positive lower bound. If $k=n/2$, one may choose
\begin{equation}
\gamma=1-\frac1{2q},
\label{eq:gamma-critical}
\end{equation}
so that $q(1-\gamma)=1/2$.
\end{lemma}

\begin{proof}
Let $C_*>0$ dominate the coefficient of the linear error $a$ in \eqref{eq:LG-general}.  By \eqref{eq:Q-lower},
\[
Q\ge\frac{k}{n}f_0a.
\]
Choose $q>1$ so large that
\[
kf_0+\frac{q-1}{2}\frac{k}{n}f_0-C_*\ge1.
\]
Then
\begin{align*}
kfa+(q-1)Q-C_*a
&=\left(kfa+\frac{q-1}{2}Q-C_*a\right)
 +\frac{q-1}{2}Q\\
&\ge a+\frac{q-1}{2}Q.
\end{align*}
Thus \eqref{eq:linear-absorb} holds with $c_0=1$ and $\mu=(q-1)/2$.

Assume first that $k>n/2$.  Both upper bounds in \eqref{eq:gamma-range} are then strictly larger than one, so we may fix $\gamma>1$ in that interval.  Hence $\theta_1=\gamma-1>0$.  For $i\ge2$, the cone estimate \eqref{eq:cone-basic} gives
\[
\frac{\kappa_i}{a}>-\frac{n-k}{k}.
\]
Since $t\mapsto(1+t)/(1-t)$ is increasing for $t<1$ and $\kappa_i<a$, we obtain
\[
\theta_i=\frac{1+\kappa_i/a}{1-\kappa_i/a}
>\frac{1-(n-k)/k}{1+(n-k)/k}
=\frac{2k-n}{n}>0.
\]
Thus all $\theta_i$ have a positive lower bound depending only on $n,k$ and the chosen $\gamma$.

If $k=n/2$, both upper endpoints in \eqref{eq:gamma-range} equal one.  Once $q>1$ has been fixed as above, set
\[
\gamma=1-\frac1{2q}.
\]
Then $0<\gamma<1$, so it is admissible, and
\[
q(1-\gamma)=\frac12.
\]
The lemma follows.
\end{proof}

For the critical case in \cref{lem:weights-general} and for the parameters in \cref{lem:weights-position}, define the good and bad spectral sets by
\begin{equation*}
\cG=\{i\ge2:\theta_i\ge1/2\},\qquad
\cB=\{1,\ldots,n\}\setminus\cG.
\end{equation*}

\begin{lemma}\label{lem:good-bad}
In either of the two settings just described, there are constants $c_G,C_B>0$ such that
\begin{equation}
\sum_iF^{ii}(\theta_i\ell_i^2+q\omega_i^2)
\ge c_G\sum_{i\in\cG}F^{ii}G_i^2
-C_B\sum_{i\in\cB}F^{ii}G_i^2.
\label{eq:goodbad-gradient}
\end{equation}
Moreover,
\begin{equation}
Q\ge\frac{a^2}{9}\sum_{i\in\cB}F^{ii}.
\label{eq:Q-bad}
\end{equation}
When $k>n/2$ in \cref{thm:gradient}, every direction is good and
\begin{equation}
\sum_iF^{ii}(\theta_i\ell_i^2+q\omega_i^2)
\ge c_G\sum_iF^{ii}G_i^2.
\label{eq:all-good}
\end{equation}
\end{lemma}

\begin{proof}
Two elementary square completions will be used.  If $\theta>0$, then
\[
\theta X^2+qY^2
=\frac{\theta}{1+q\theta}(X+qY)^2
 +\frac{q}{1+q\theta}(\theta X-Y)^2.
\]
Hence, whenever $\theta\ge1/2$,
\[
\theta X^2+qY^2\ge c(q)(X+qY)^2,
\qquad
c(q)=\frac{1/2}{1+q/2}>0.
\]
If $\theta=-d<0$ and $qd<1$, then
\[
-dX^2+qY^2
=-\frac{d}{1-qd}(X+qY)^2
 +\frac{q}{1-qd}(dX+Y)^2,
\]
so
\[
-dX^2+qY^2\ge -\frac{d}{1-qd}(X+qY)^2.
\]
Finally, if $0\le\theta<1/2$, the left-hand side is nonnegative, and therefore it is certainly bounded below by $-C(X+qY)^2$.

Apply these identities with $X=\ell_i$ and $Y=\omega_i$, so that $X+qY=G_i$.  It remains only to check that the denominator $1-qd$ stays uniformly positive whenever $\theta_i<0$.  In the critical gradient-dependent case, the only negative coefficient is $\theta_1=\gamma-1=-1/(2q)$, because \eqref{eq:cone-basic} with $k=n/2$ gives $\theta_i>0$ for $i\ge2$.  Thus $qd=1/2$.  In the position-dependent case, for $i=1$ one has
\[
d=1-\gamma=\frac{n-k}{n},
\qquad
1-qd=\frac{k}{2n}>0
\]
by \eqref{eq:position-margins}.  If $i\ge2$ and $\theta_i<0$, this can occur only when $k<n/2$.  The cone estimate gives
\[
\theta_i> -\frac{n-2k}{n},
\]
so $d<(n-2k)/n$ and \eqref{eq:position-extra-margin} yields a uniform positive lower bound for $1-qd$.  Hence the negative coefficient in the preceding square completion is uniformly bounded.  Summing with weights $F^{ii}>0$ proves \eqref{eq:goodbad-gradient}.

It remains to prove \eqref{eq:Q-bad}.  By definition $1\in\cB$ and $\kappa_1=a$, so the $i=1$ contribution to $Q$ is at least $(a^2/9)F^{11}$.  If $i\in\cB$ and $i\ge2$, then $\theta_i<1/2$, and therefore
\[
\frac{a+\kappa_i}{a-\kappa_i}<\frac12.
\]
Since $a-\kappa_i>0$, this is equivalent to
\[
2a+2\kappa_i<a-\kappa_i,
\qquad\text{hence}\qquad
\kappa_i<-\frac a3.
\]
Thus $\kappa_i^2>a^2/9$ for every bad index, and
\[
Q=\sum_iF^{ii}\kappa_i^2
\ge\frac{a^2}{9}\sum_{i\in\cB}F^{ii}.
\]

In the supercritical gradient-dependent case, \cref{lem:weights-general} gives a uniform lower bound $\theta_i\ge\delta>0$ for every $i$.  The positive square completion above then yields
\[
\theta_i\ell_i^2+q\omega_i^2\ge c(\delta,q)G_i^2
\]
for all $i$, and summation proves \eqref{eq:all-good}.
\end{proof}

\section{Proofs of the main theorems}

\subsection{Two-surface localization and comparison geometry}

We adapt the quadratic product-space localization introduced by Qiu and Yan \cite[Section~4.2, proof of Theorem~4.5]{QiuYan2026}.  The localization scale is determined by the height of the comparison gap.  We apply the same construction to a truncated gap so that no boundary bound for the second derivatives of $u$ is needed.

Fix $0<\varepsilon<\norm{\rho}_{L^\infty(\Omega)}$ and define
\[
\Omega_\varepsilon=\{\rho>\varepsilon\},\qquad h=\rho-\varepsilon.
\]
Since $\rho=0$ on $\partial\Omega$, the set $\Omega_\varepsilon$ is compactly contained in $\Omega$. Let $\beta>2$ be fixed later, and choose $x_0\in\Omega_\varepsilon$ such that
\begin{equation}
h(x_0)^\beta Z(x_0)=\max_{\Omega_\varepsilon}h^\beta Z.
\label{eq:x0-max}
\end{equation}
Write
\[
r=h(x_0)>0.
\]
After increasing the prescribed slope bound, assume $M\ge1$. Let $E_r$ be the component containing $x_0$ of
\[
\{x\in\Omega_\varepsilon:r/2<h(x)<2r\}.
\]
Since $\abs{D\rho}\le M$ and $\rho=0$ on $\partial\Omega$,
\begin{equation}
\dist(x,\partial\Omega)\ge\frac{\rho(x)}{M}>\frac{r}{2M}\qquad(x\in E_r).
\label{eq:dist-boundary}
\end{equation}
Therefore every $y$ with $\abs{y-x}<r/(2M)$ belongs to $\Omega$.

Define
\begin{equation*}
D_r=\left\{(x,y):x\in E_r,\ \abs{y-x}<\frac{r}{2M}\right\}
\end{equation*}
and
\begin{equation}
\eta(x,y)=w(y)-u(x)-\varepsilon-\frac r2-\frac r{32}-\frac{8M^2}{r}\abs{y-x}^2.
\label{eq:eta}
\end{equation}
Let $D_r^+$ be the connected component of $D_r\cap\{\eta>0\}$ that contains $(x_0,x_0)$. On $D_r^+$, consider
\begin{equation}
P(x,y)=G(x)+\frac\beta2\log\eta(x,y).
\label{eq:P}
\end{equation}

\begin{lemma}\label{lem:interior-max}
The function $P$ attains its maximum at an interior point $(x_*,y_*)\in D_r$ with $\eta(x_*,y_*)>0$. Moreover,
\begin{equation}
0<\eta(x,y)\le h(x)-\frac r2\le\frac{3r}{2}
\label{eq:eta-bounds}
\end{equation}
on $D_r^+$, and
\begin{equation}
Dw(y_*)=\frac{16M^2}{r}(y_*-x_*),\qquad
\abs{y_*-x_*}\le\frac{r}{16M}.
\label{eq:ycritical}
\end{equation}
\end{lemma}

\begin{proof}
At the diagonal pair $(x_0,x_0)$,
\[
\eta(x_0,x_0)=r-\frac r2-\frac r{32}=\frac{15r}{32}>0,
\]
so $D_r^+$ is nonempty and the weighted quantity used below is positive somewhere.

For \eqref{eq:eta-bounds},  Fix $(x,y)\in D_r$.  By \eqref{eq:dist-boundary}, the Euclidean ball $B_{r/(2M)}(x)$ lies in $\Omega$, so the line segment from $x$ to $y$ remains inside $\Omega$.  The gradient bound for $w$ gives
\[
w(y)-w(x)\le M\abs{y-x}.
\]
Since $h(x)=w(x)-u(x)-\varepsilon$, putting $t=\abs{y-x}$ gives
\begin{align*}
\eta(x,y)
&\le h(x)+Mt-\frac r2-\frac r{32}-\frac{8M^2}{r}t^2\\
&\le h(x)-\frac r2,
\end{align*}
because the quadratic polynomial $Mt-8M^2t^2/r$ has maximum $r/32$, attained at $t=r/(16M)$.  Since $h<2r$ on $E_r$, this also gives $\eta<3r/2$.

It remains to locate the maximum in the interior.  The closure of $E_r$ is compactly contained in $\Omega$.  Indeed, on $\overline E_r$ one has $\rho=\varepsilon+h\ge r/2$, while $\abs{D\rho}\le M$ and $\rho=0$ on $\partial\Omega$ imply
\[
\dist(x,\partial\Omega)\ge\frac{\rho(x)}{M}\ge\frac{r}{2M}.
\]
Thus the closed set
\[
\overline D_r
=\left\{(x,y):x\in\overline E_r,
\ \abs{y-x}\le\frac{r}{2M}\right\}
\]
is compact in $\Omega\times\Omega$.  The function
\[
\Psi(x,y)=Z(x)\max\{\eta(x,y),0\}^{\beta/2}
\]
is continuous on $\overline D_r$ and satisfies $\Psi(x_0,x_0)>0$, so it attains a positive maximum.

No boundary point can be a maximizer.  If $h=r/2$, then the first estimate above gives $\eta\le0$.  If $\abs{y-x}=r/(2M)$, then
\[
Mt=\frac r2,
\qquad
\frac{8M^2}{r}t^2=2r,
\]
and hence, using $h\le2r$ on $\overline E_r$,
\[
\eta\le h-\frac r2-\frac r{32}+\frac r2-2r\le-\frac r{32}<0.
\]
The remaining part of $\partial E_r$ is $h=2r$.  By the maximality of $x_0$ in \eqref{eq:x0-max},
\[
Z(x)\le2^{-\beta}Z(x_0).
\]
Moreover $\eta(x,y)\le3r/2$, while $\eta(x_0,x_0)=15r/32$.  Therefore
\[
\frac{\Psi(x,y)}{\Psi(x_0,x_0)}
\le 2^{-\beta}\left(\frac{16}{5}\right)^{\beta/2}
=\left(\frac45\right)^{\beta/2}<1.
\]
Hence the positive maximum lies at an interior point $(x_*,y_*)$ with $\eta(x_*,y_*)>0$.  Maximizing $\Psi$ is equivalent there to maximizing $P=\log Z+(\beta/2)\log\eta$.

Finally differentiate $P$ in the $y$ variables.  Since $G$ depends only on $x$, the $y$-gradient of $\eta$ must vanish at the maximum:
\[
0=D_y\eta
 =Dw(y_*)-\frac{16M^2}{r}(y_*-x_*).
\]
Hence the first identity in \eqref{eq:ycritical} holds.  Since $\abs{Dw(y_*)}\le M$,
\[
\abs{y_*-x_*}
=\frac{r}{16M^2}\abs{Dw(y_*)}
\le\frac{r}{16M},
\]
which is the second identity.
\end{proof}

At the product-space maximum, compare the tangent spaces of $M_u$ at $(x_*,u(x_*))$ and $M_w$ at $(y_*,w(y_*))$. Let
\[
c=\inn{\nu_u}{\nu_w}.
\]
The downward graph normals lie in a fixed spherical cap, and therefore
\begin{equation}
1+c\ge\frac{2}{1+M^2}>0.
\label{eq:angle-lower}
\end{equation}
The natural rotation from $T_{(x_*,u(x_*))}M_u$ to $T_{(y_*,w(y_*))}M_w$ is the isometry
\begin{equation}
RZ=Z-\frac{\inn{\nu_w}{Z}}{1+c}(\nu_u+\nu_w).
\label{eq:natural-rotation}
\end{equation}
It is the restriction of the ambient rotation that sends $\nu_u$ to $\nu_w$ in their two-plane and fixes its orthogonal complement.

Let $\{e_i\}$ be an orthonormal principal frame for $A_u$ at $(x_*,u(x_*))$, and set $e_i'=Re_i$. For each $i$, take unit-speed geodesics on $M_u$ and $M_w$ with initial velocities $e_i$ and $e_i'$, respectively. Their horizontal projections are denoted by $x_i(s)$ and $y_i(s)$. We write
\[
\eta_i=\left.\frac{d}{ds}\eta(x_i(s),y_i(s))\right|_{s=0},\qquad
\eta_i''=\left.\frac{d^2}{ds^2}\eta(x_i(s),y_i(s))\right|_{s=0}.
\]

The following lemma is the $\sigma_k$ version of the paired-geodesic comparison used in \cite[Section~4.2]{QiuYan2026}.  The natural rotation identifies the two tangent spaces, and the mixed G\aa rding inequality controls the contribution of the upper graph using only its $\Gamma_k$-admissibility.

\begin{lemma}\label{lem:paired-second}
At the interior maximum of $P$, let
\begin{equation*}
S=\sum_iF^{ii}\eta_i^2.
\end{equation*}
Then
\begin{equation}
\abs{\eta_i}\le C(M),
\label{eq:eta-first-bound}
\end{equation}
and
\begin{equation}
\sum_iF^{ii}\eta_i''\ge -C-C\frac{S}{r}.
\label{eq:eta-second-bound}
\end{equation}
The constants are independent of $D^2w(y_*)$ and of $\inf\sigma_k(A_w)$.
\end{lemma}

\begin{proof}
Write the paired unit tangent vectors in graph coordinates as
\[
e_i=(X_i,Du(x_*)\cdot X_i),
\qquad
e_i'=(Y_i,Dw(y_*)\cdot Y_i),
\]
and set $B_i=Y_i-X_i$.  Since $D_y\eta(x_*,y_*)=0$, \eqref{eq:ycritical} gives
\[
Dw(y_*)=\frac{16M^2}{r}(y_*-x_*).
\]
Differentiating $\eta$ once along the paired geodesics therefore yields
\begin{align*}
\eta_i
&=Dw(y_*)\cdot Y_i-Du(x_*)\cdot X_i
 -\frac{16M^2}{r}(y_*-x_*)\cdot(Y_i-X_i)\\
&=(Dw(y_*)-Du(x_*))\cdot X_i.
\end{align*}
Because $\nu_w=(Dw,-1)/W_w$,
\[
W_w\inn{\nu_w}{e_i}
=Dw\cdot X_i-Du\cdot X_i,
\]
and hence
\[
\eta_i=W_w\inn{\nu_w}{e_i}.
\]
Thus $\abs{\eta_i}\le W_w\le\sqrt{1+M^2}$, which proves \eqref{eq:eta-first-bound}.

Next use the explicit formula for the natural rotation.  Since $e_i'=Re_i$ and $e_i\perp\nu_u$,
\[
e_i'-e_i
=-\frac{\inn{\nu_w}{e_i}}{1+c}(\nu_u+\nu_w)
=-\frac{\eta_i}{W_w(1+c)}(\nu_u+\nu_w).
\]
The vector $B_i$ is the horizontal projection of $e_i'-e_i$, so
\begin{align}
\abs{B_i}^2
&\le\abs{e_i'-e_i}^2
 =\frac{\eta_i^2}{W_w^2(1+c)^2}\abs{\nu_u+\nu_w}^2
 =\frac{2\eta_i^2}{W_w^2(1+c)}
\le C(M)\eta_i^2.
\label{eq:Bi-bound}
\end{align}
Here we used \eqref{eq:angle-lower}.

For the second derivative of $\eta$, put
\[
\alpha=\frac{8M^2}{r},
\qquad p=y_*-x_*.
\]
Let
\[
b_i=A_w(e_i',e_i').
\]
Because the curves on the two graphs are unit-speed geodesics, their ambient accelerations at $s=0$ are
\[
-\kappa_i\nu_u,
\qquad
-b_i\nu_w.
\]
Taking horizontal components gives
\[
x_i''(0)=-\frac{\kappa_i}{W_u}Du(x_*),
\qquad
y_i''(0)=-\frac{b_i}{W_w}Dw(y_*),
\]
while the second derivative of the vertical difference is
\[
\frac{d^2}{ds^2}\big(w(y_i(s))-u(x_i(s))\big)\Big|_{s=0}
=\frac{b_i}{W_w}-\frac{\kappa_i}{W_u}.
\]
The second derivative of the quadratic cost is
\[
-2\alpha\left(\abs{B_i}^2+p\cdot(y_i''(0)-x_i''(0))\right).
\]
Since the $y$-critical relation is $2\alpha p=Dw(y_*)$, the acceleration part of the cost becomes
\begin{align*}
-2\alpha p\cdot(y_i''-x_i'')
&=-Dw\cdot\left(-\frac{b_i}{W_w}Dw
 +\frac{\kappa_i}{W_u}Du\right)\\
&=\frac{b_i}{W_w}\abs{Dw}^2
 -\frac{\kappa_i}{W_u}Dw\cdot Du.
\end{align*}
Combining these formulas and using
\[
W_w^2=1+\abs{Dw}^2,
\qquad
c=\inn{\nu_u}{\nu_w}
 =\frac{1+Du\cdot Dw}{W_uW_w},
\]
we obtain
\begin{align*}
\eta_i''
&=\frac{b_i}{W_w}-\frac{\kappa_i}{W_u}
 +\frac{b_i}{W_w}\abs{Dw}^2
 -\frac{\kappa_i}{W_u}Dw\cdot Du
 -2\alpha\abs{B_i}^2\\
&=W_w\big(b_i-c\kappa_i\big)
 -\frac{16M^2}{r}\abs{B_i}^2.
\end{align*}

Taking the $F$-trace,  Since $\{e_i\}$ diagonalizes $A_u$,
\[
\sum_iF^{ii}b_i
=D\sigma_k(A_u):R^*A_wR,
\qquad
\sum_iF^{ii}\kappa_i=kf.
\]
Therefore
\begin{equation}
\sum_iF^{ii}\eta_i''
=W_w\left\{D\sigma_k(A_u):R^*A_wR-ckf\right\}
 -\frac{16M^2}{r}\sum_iF^{ii}\abs{B_i}^2.
\label{eq:eta-second-trace}
\end{equation}
Because $R$ is an isometry and $A_w\in\Gamma_k$, the conjugate $R^*A_wR$ also lies in the matrix G\aa rding cone.  Concavity and one-homogeneity of $\sigma_k^{1/k}$ give the mixed G\aa rding inequality
\[
D\sigma_k(A_u):R^*A_wR
\ge k\sigma_k(A_u)^{(k-1)/k}\sigma_k(A_w)^{1/k}\ge0.
\]
Only the nonnegativity of this term is used; no positive lower bound for $\sigma_k(A_w)$ is required.  The remaining zero-order term satisfies
\[
-W_wckf\ge-C(n,k,M,f_1),
\]
while \eqref{eq:Bi-bound} gives
\[
-\frac{16M^2}{r}\sum_iF^{ii}\abs{B_i}^2
\ge-\frac{C(M)}{r}\sum_iF^{ii}\eta_i^2
=-C\frac{S}{r}.
\]
Combining the last three estimates proves \eqref{eq:eta-second-bound}.
\end{proof}

\begin{remark}\label{rem:natural-rotation}
The natural rotation is essential. A fixed-coordinate contraction involving $D^2w$ at the same base point can have the wrong sign and arbitrarily large magnitude even for admissible $\Gamma_2$ jets with bounded slopes; see \cite[Remark 3.3]{QiuYan2026}. No estimate of such a fixed-coordinate contraction is used.  Instead, the second variation produces the geometrically invariant expression in \eqref{eq:eta-second-trace}.
\end{remark}

\subsection{Maximum-principle closure}

All quantities in this subsection are evaluated at the product-space maximum $(x_*,y_*)$. The first derivative of \eqref{eq:P} in the paired directions yields
\begin{equation}
G_i=-\frac\beta{2\eta}\eta_i.
\label{eq:first-P}
\end{equation}
The second derivative, traced with $F^{ii}$, and \cref{lem:paired-second} yield
\begin{equation}
0\ge \cL G-C\frac\beta\eta-C\frac\beta{\eta^2}S.
\label{eq:second-P}
\end{equation}
Indeed, the term $S/(r\eta)$ is bounded by $3S/(2\eta^2)$ because $\eta\le3r/2$.

\begin{proposition}\label{prop:contact}
Under the assumptions of either \cref{thm:main} or \cref{thm:gradient}, $\beta$ can be fixed in terms of the corresponding theorem data so that
\begin{equation}
a(x_*)\eta(x_*,y_*)\le C,
\label{eq:contact}
\end{equation}
with $C$ independent of $\varepsilon$ and $r$.
\end{proposition}

\begin{proof}
If \eqref{eq:small-level} fails at $(x_*,y_*)$, then \eqref{eq:smalllevel-fail} gives a fixed upper bound for $a(x_*)$.  Since \eqref{eq:eta-bounds} implies $\eta\le3r/2\le(3/2)\norm{\rho}_{L^\infty}$, \eqref{eq:contact} follows immediately.  We may therefore assume throughout the remainder of the proof that \eqref{eq:small-level} holds.

At the product-space maximum, differentiating \eqref{eq:P} in the paired directions gives \eqref{eq:first-P}; hence
\[
G_i^2=\frac{\beta^2}{4\eta^2}\eta_i^2.
\]
Combine this relation with \eqref{eq:second-P} in the three parameter regimes.

\smallskip
\noindent\emph{Case 1: $f=f(x,u)$ and $2\le k<n$.}
By \eqref{eq:LG-position} and \cref{lem:good-bad},
\[
\cL G\ge kf_0a+(q-1)Q
+c_G\sum_{i\in\cG}F^{ii}G_i^2
-C_B\sum_{i\in\cB}F^{ii}G_i^2-C.
\]
Set
\[
S_G=\sum_{i\in\cG}F^{ii}\eta_i^2,
\qquad
S_B=\sum_{i\in\cB}F^{ii}\eta_i^2.
\]
Substitution into \eqref{eq:second-P} yields
\[
0\ge kf_0a+(q-1)Q
+\left(\frac{c_G\beta^2}{4}-C\beta\right)\frac{S_G}{\eta^2}
-C(\beta^2+\beta)\frac{S_B}{\eta^2}
-C\frac\beta\eta-C.
\]
Choose $\beta=\beta(n,k,M)>2$ so large that the coefficient of $S_G/\eta^2$ is nonnegative.  By \eqref{eq:eta-first-bound} and \eqref{eq:Q-bad},
\[
S_B\le C(M)\sum_{i\in\cB}F^{ii}
\le\frac{C(M)}{a^2}Q.
\]
There is a constant $K=K(n,k,M)$ such that, whenever $a\eta\ge K\beta$, the bad-direction term is bounded below by $-(q-1)Q/2$.  Indeed,
\[
C(\beta^2+\beta)\frac{S_B}{\eta^2}
\le C\frac{\beta^2}{a^2\eta^2}Q,
\]
and the coefficient on the right is as small as desired once $K$ is chosen large.  If $a\eta<K\beta$, \eqref{eq:contact} already holds.  Otherwise we may absorb the bad term and obtain
\[
0\ge kf_0a+\frac{q-1}{2}Q-C\frac\beta\eta-C.
\]
Dropping the nonnegative $Q$ term and multiplying by $\eta$ gives
\[
kf_0a\eta\le C\beta+C\eta.
\]
Since $\eta\le(3/2)\norm\rho_{L^\infty}$ and $\beta$ is fixed, this proves \eqref{eq:contact}.

\smallskip
\noindent\emph{Case 2: $f=f(x,u,Du)$ and $k>n/2$.}
By \cref{lem:weights-general,lem:good-bad}, after absorbing the linear loss as in \eqref{eq:linear-absorb},
\[
\cL G\ge c_0a+\mu Q+c_G\sum_iF^{ii}G_i^2
+D_\nu\whf(\nabla G)-C.
\]
The first derivative relation \eqref{eq:first-P} gives
\[
\nabla G=-\frac\beta{2\eta}\sum_i\eta_i e_i.
\]
Thus, using the $C^1$ bound for $\whf$ and \eqref{eq:eta-first-bound},
\[
D_\nu\whf(\nabla G)
=-\frac\beta{2\eta}\sum_iD_\nu\whf(e_i)\eta_i
\ge-C\frac\beta\eta.
\]
Moreover
\[
\sum_iF^{ii}G_i^2=\frac{\beta^2}{4\eta^2}S.
\]
Substitution into \eqref{eq:second-P} gives
\[
0\ge c_0a+\mu Q
+\left(\frac{c_G\beta^2}{4}-C\beta\right)\frac{S}{\eta^2}
-C\frac\beta\eta-C.
\]
Choose $\beta>2$ so that the coefficient of $S/\eta^2$ is nonnegative.  Dropping the nonnegative terms then yields
\[
0\ge c_0a-C\frac\beta\eta-C.
\]
As in Case 1, multiplication by $\eta$ and the bound for $\eta$ prove \eqref{eq:contact}.

\smallskip
\noindent\emph{Case 3: $f=f(x,u,Du)$ and $k=n/2$.}
The normal-drift term is estimated exactly as in Case 2.  Using \eqref{eq:goodbad-gradient} in place of \eqref{eq:all-good}, we obtain
\[
0\ge c_0a+\mu Q
+\left(\frac{c_G\beta^2}{4}-C\beta\right)\frac{S_G}{\eta^2}
-C(\beta^2+\beta)\frac{S_B}{\eta^2}
-C\frac\beta\eta-C.
\]
Fix $\beta>2$ so that the good-direction coefficient is nonnegative.  The estimate
\[
S_B\le\frac{C(M)}{a^2}Q
\]
is the same as in Case 1.  If $a\eta$ is already bounded by $K\beta$, there is nothing to prove.  Otherwise choose $K$ large enough that the bad-direction contribution is bounded below by $-\mu Q/2$.  We then get
\[
0\ge c_0a+\frac\mu2Q-C\frac\beta\eta-C,
\]
and the same final multiplication by $\eta$ proves \eqref{eq:contact}.
\end{proof}

\begin{proof}[Proof of \cref{thm:main,thm:gradient}]
Fix the value of $q$ prescribed by the corresponding case above, and choose $\beta>2$ as in \cref{prop:contact}.  Since $1\le W_u\le\sqrt{1+M^2}$,
\[
a\le Z=aW_u^q\le C(M,q)a.
\]
At the product-space maximum, \cref{prop:contact} gives $a\eta\le C$, while \eqref{eq:eta-bounds} gives $\eta\le3r/2$.  Hence
\begin{align*}
Z(x_*)\eta(x_*,y_*)^{\beta/2}
&\le C a(x_*)\eta(x_*,y_*)^{\beta/2}\\
&=C\big(a\eta\big)(x_*,y_*)
  \eta(x_*,y_*)^{\beta/2-1}\\
&\le C r^{\beta/2-1}.
\end{align*}
The exponent is positive because $\beta>2$.

The maximality of $P$ and the diagonal value
\[
\eta(x_0,x_0)=\frac{15r}{32}
\]
therefore imply
\[
Z(x_0)\left(\frac{15r}{32}\right)^{\beta/2}
\le Z(x_*)\eta(x_*,y_*)^{\beta/2}
\le Cr^{\beta/2-1}.
\]
After cancelling $r^{\beta/2}$ we obtain
\[
rZ(x_0)\le C.
\]
By the definition of $x_0$ in \eqref{eq:x0-max},
\begin{align*}
\sup_{\Omega_\varepsilon}h^\beta a
&\le\sup_{\Omega_\varepsilon}h^\beta Z
 =r^\beta Z(x_0)\\
&\le Cr^{\beta-1}
 \le C\norm{\rho}_{L^\infty(\Omega)}^{\beta-1}.
\end{align*}
All constants are independent of $\varepsilon$.

Fix $x\in\Omega$.  Since $\rho(x)>0$, for every sufficiently small $\varepsilon>0$ one has $x\in\Omega_\varepsilon$, and $h(x)=\rho(x)-\varepsilon\to\rho(x)$.  Letting $\varepsilon\downarrow0$ in the preceding estimate gives
\[
\rho(x)^\beta\kappa_1[u](x)\le C.
\]
Taking the supremum over $x$ yields
\[
\sup_\Omega\rho^\beta\kappa_1[u]\le C.
\]
Finally, because $k\ge2$ and $A_u\in\Gamma_k\subset\Gamma_2$, \cref{lem:cone} gives
\[
\abs{A_u}<H\le n\kappa_1[u].
\]
Therefore
\[
\sup_\Omega\rho^\beta\abs{A_u}\le C.
\]
The preceding proof also shows the stated dependence of the constants.  In particular, the comparison graph enters at second order only in \cref{lem:paired-second}, where the mixed G\aa rding inequality uses merely $A_w\in\Gamma_k$ and never a quantitative lower bound for $\sigma_k(A_w)$ or an upper bound for $D^2w$.
\end{proof}

\subsection{Remarks on the assumptions and the range of the method}

The argument uses essentially positivity of $f$, full $\Gamma_k$-admissibility of both graphs, a strict interior gap that vanishes on the boundary, and uniform slope bounds. Merely assuming $\sigma_k(A_w)>0$ is not enough for the mixed G\aa rding step, because the full matrix must remain in the G\aa rding cone.

For a general $f(x,u,Du)$, the current argument closes in the range $k\ge n/2$. When $k<n/2$, Yan's admissible parameter satisfies $\gamma<1$, while the coefficient $q$ needed to absorb a large normal-Hessian loss cannot be increased without destroying the first-direction square completion. The position-dependent theorem avoids this conflict because the $C^{1,1}$ source error is bounded rather than linear in $a$.


\begin{thebibliography}{99}
\bibitem{CNS1984}
L.~A.~Caffarelli, L.~Nirenberg, and J.~Spruck,
\newblock The Dirichlet problem for nonlinear second-order elliptic equations. I. Monge--Amp\`ere equation,
\newblock \emph{Comm. Pure Appl. Math.} \textbf{37} (1984), 369--402.
\bibitem{CNS1985}
L.~A.~Caffarelli, L.~Nirenberg, and J.~Spruck,
\newblock The Dirichlet problem for nonlinear second-order elliptic equations. III. Functions of the eigenvalues of the Hessian,
\newblock \emph{Acta Math.} \textbf{155} (1985), 261--301.
\bibitem{CNS1986}
L.~A.~Caffarelli, L.~Nirenberg, and J.~Spruck,
\newblock Nonlinear second order elliptic equations. IV. Starshaped compact Weingarten hypersurfaces,
\newblock in \emph{Current Topics in Partial Differential Equations}, Kinokuniya, Tokyo, 1986, 1--26.
\bibitem{CNS1988}
L.~A.~Caffarelli, L.~Nirenberg, and J.~Spruck,
\newblock Nonlinear second-order elliptic equations. V. The Dirichlet problem for Weingarten hypersurfaces,
\newblock \emph{Comm. Pure Appl. Math.} \textbf{41} (1988), 47--70.
\bibitem{ChenTuXiang2021}
L.~Chen, Q.~Tu, and N.~Xiang,
\newblock Pogorelov type estimates for a class of Hessian quotient equations,
\newblock \emph{J. Differential Equations} \textbf{282} (2021), 272--284.
\bibitem{ChenLuTuXiang2021}
X.~Chen, W.~Lu, Q.~Tu, and N.~Xiang,
\newblock Pogorelov estimates for a class of fully nonlinear equations,
\newblock \emph{Nonlinear Anal.} \textbf{212} (2021), Paper No.~112482.
\bibitem{ChouWang2001}
K.-S.~Chou and X.-J.~Wang,
\newblock A variational theory of the Hessian equation,
\newblock \emph{Comm. Pure Appl. Math.} \textbf{54} (2001), 1029--1064.
\bibitem{Fan2026}
Z.~Fan,
\newblock Hessian estimates for the $\sigma_2$ equation with variable right-hand side terms in dimension $4$,
\newblock \emph{Adv. Math.} \textbf{494} (2026), Paper No.~110953.
\bibitem{Garding1959}
L.~G\aa rding,
\newblock An inequality for hyperbolic polynomials,
\newblock \emph{J. Math. Mech.} \textbf{8} (1959), 957--965.
\bibitem{GuanQiu2019}
P.~Guan and G.~Qiu,
\newblock Interior $C^2$ regularity of convex solutions to prescribing scalar curvature equations,
\newblock \emph{Duke Math. J.} \textbf{168} (2019), 1641--1663.
\bibitem{GuanRenWang2015}
P.~Guan, C.~Ren, and Z.~Wang,
\newblock Global $C^2$-estimates for convex solutions of curvature equations,
\newblock \emph{Comm. Pure Appl. Math.} \textbf{68} (2015), 1287--1325.
\bibitem{Ivochkina1990}
N.~M.~Ivochkina,
\newblock Solution of the Dirichlet problem for curvature equations of order $m$,
\newblock \emph{Math. USSR-Sb.} \textbf{67} (1990), 317--339.
\bibitem{JiaoJiao2024}
H.~Jiao and Y.~Jiao,
\newblock The Pogorelov estimates for degenerate curvature equations,
\newblock \emph{Int. Math. Res. Not. IMRN} \textbf{2024} (2024), no.~18, 12504--12529.
\bibitem{LiRenWang2019}
C.~Li, C.~Ren, and Z.~Wang,
\newblock Curvature estimates for convex solutions of some fully nonlinear Hessian-type equations,
\newblock \emph{Calc. Var. Partial Differential Equations} \textbf{58} (2019), Paper No.~188.
\bibitem{LiuRen2023}
Y.~Liu and C.~Ren,
\newblock Pogorelov type $C^2$ estimates for Sum Hessian equations and a rigidity theorem,
\newblock \emph{J. Funct. Anal.} \textbf{284} (2023), Paper No.~109726.
\bibitem{Lu2023}
S.~Lu,
\newblock Curvature estimates for semi-convex solutions of Hessian equations in hyperbolic space,
\newblock \emph{Calc. Var. Partial Differential Equations} \textbf{62} (2023), Paper No.~257.
\bibitem{LuTsai2026}
S.~Lu and Y.-L.~Tsai,
\newblock A simple proof of curvature estimates for the $n-1$ Hessian equation,
\newblock \emph{Proc. Amer. Math. Soc.} \textbf{154} (2026), 893--904.
\bibitem{Pogorelov1978}
A.~V.~Pogorelov,
\newblock \emph{The Minkowski Multidimensional Problem},
\newblock V.~H. Winston \& Sons, Washington, D.C.; Halsted Press, New York, 1978.
\bibitem{Qiu2024Curvature}
G.~Qiu,
\newblock Interior curvature estimates for hypersurfaces of prescribing scalar curvature in dimension three,
\newblock \emph{Amer. J. Math.} \textbf{146} (2024), 579--605.
\bibitem{Qiu2024Hessian}
G.~Qiu,
\newblock Interior Hessian estimates for $\sigma_2$ equations in dimension three,
\newblock \emph{Front. Math.} \textbf{19} (2024), 577--598.
\bibitem{QiuYan2026}
G.~Qiu and J.~Yan,
\newblock Interior curvature estimates for the graphical scalar curvature equation in all dimensions,
\newblock arXiv:2609.02581v2 [math.AP], 2026.
\bibitem{RenWang2019}
C.~Ren and Z.~Wang,
\newblock On the curvature estimates for Hessian equations,
\newblock \emph{Amer. J. Math.} \textbf{141} (2019), 1281--1315.
\bibitem{RenWang2023}
C.~Ren and Z.~Wang,
\newblock The global curvature estimate for the $n-2$ Hessian equation,
\newblock \emph{Calc. Var. Partial Differential Equations} \textbf{62} (2023), Paper No.~239.
\bibitem{ShankarYuan2020}
R.~Shankar and Y.~Yuan,
\newblock Hessian estimate for semiconvex solutions to the $\sigma_2$ equation,
\newblock \emph{Calc. Var. Partial Differential Equations} \textbf{59} (2020), Paper No.~30.
\bibitem{ShankarYuan2025}
R.~Shankar and Y.~Yuan,
\newblock Hessian estimates for the $\sigma_2$ equation in dimension four,
\newblock \emph{Ann. of Math.} \textbf{201} (2025), 489--513.
\bibitem{ShengTrudingerWang2004}
W.~Sheng, N.~S.~Trudinger, and X.-J.~Wang,
\newblock Convex hypersurfaces of prescribed Weingarten curvatures,
\newblock \emph{Comm. Anal. Geom.} \textbf{12} (2004), 213--232.
\bibitem{ShengUrbasWang2004}
W.~Sheng, J.~Urbas, and X.-J.~Wang,
\newblock Interior curvature bounds for a class of curvature equations,
\newblock \emph{Duke Math. J.} \textbf{123} (2004), 235--264.
\bibitem{SpruckXiao2017}
J.~Spruck and L.~Xiao,
\newblock A note on star-shaped compact hypersurfaces with prescribed scalar curvature in space forms,
\newblock \emph{Rev. Mat. Iberoam.} \textbf{33} (2017), 547--554.
\bibitem{Trudinger1990}
N.~S.~Trudinger,
\newblock The Dirichlet problem for the prescribed curvature equations,
\newblock \emph{Arch. Rational Mech. Anal.} \textbf{111} (1990), 153--179.
\bibitem{Urbas1990}
J.~I.~E.~Urbas,
\newblock On the existence of nonclassical solutions for two classes of fully nonlinear elliptic equations,
\newblock \emph{Indiana Univ. Math. J.} \textbf{39} (1990), 355--382.
\bibitem{Urbas2000}
J.~Urbas,
\newblock An interior curvature bound for hypersurfaces of prescribed $k$-th mean curvature,
\newblock \emph{J. Reine Angew. Math.} \textbf{519} (2000), 41--57.
\bibitem{WarrenYuan2009}
M.~Warren and Y.~Yuan,
\newblock Hessian estimates for the $\sigma_2$ equation in dimension three,
\newblock \emph{Comm. Pure Appl. Math.} \textbf{62} (2009), 305--321.
\bibitem{Yan2026}
J.~Yan,
\newblock Global curvature estimates for $\sigma_k$ curvature equations with $k\ge n/2$,
\newblock arXiv:2608.25665v1 [math.AP], 2026.
\end{thebibliography}
\end{document}